\documentclass[12pt]{article}
\usepackage{amssymb,amsmath,amsthm,amscd}

\newtheorem{df}{Definition}
\newtheorem{lem}{Lemma}
\newtheorem{teo}{Theorem}
\newtheorem{prop} {Proposition}

\newcommand{\Z}{\mathbb{Z}}
\newcommand{\C}{\mathbb{C}}
\newcommand{\R}{\mathbb{R}}
\newcommand{\nn}{\newline\noindent}

\begin{document}
\title{{The disk property of coverings of 1-convex surfaces}
\author{M. Col\c toiu and C. Joi\c ta }}

\date{}
\maketitle

\begin{abstract}
Let $X$ be an 1-convex surface and $p:\tilde X\to X$ an (unbranched) covering map. We prove that
if $\tilde X$ does not contain an infinite Nori string of rational curves then $\tilde X$ satisfies the 
discrete disk property.
\end{abstract}
\section{Introduction}

Let $X$ be a 1-convex surface, i.e. a two dimensional complex manifold which is strongly pseudoconvex. 
Then $X$ is a proper 
modification of a 2-dimensional normal Stein space at a finite set of points. Let  $p:\tilde X\to X$ 
be an 
(unbranched) covering map. We are interested in this note to study the convexity properties 
of $\tilde X$.

It was remarked in \cite{Co} that in general $\tilde X$ is not weakly 1-complete (i.e. it might be possible 
that there is no plurisubharmonic exhaustion function on $\tilde X$). This is due to the fact that 
$\tilde X$ might contain an infinite Nori string (necklace), that is a 1-dimensional connected complex 
subspace which is non-compact and has infinitely many compact irreducible components. 
However the main result in  \cite{Co} shows that $\tilde X$ can be exhausted by a sequence of relatively
compact strongly pseudoconvex domains with smooth boundary ($\tilde X$ is $p_3$-convex in
the terminology of \cite{DG}). In particular $\tilde X$ satisfies the continuous Kontinuit\"atssatz (the
continuous disk property).

In this paper we investigate the discrete disk property for $\tilde X$. The main result can be stated as 
follows (see  Theorem \ref{main}): \textit{If $\tilde X$ does not contain an infinite Nori string
of rational curves then $\tilde X$ satisfies the discrete disk property}. 

It should be noted that the discrete disk property is a much stronger condition than the continuous one.

\section{Preliminaries}

All complex spaces are assumed of bounded dimension and countable at infinity.

Let $X$ be a complex manifold. We recall that $X$ is said to be 1-convex if there exists a 
${\mathcal C}^\infty$ function $\phi:X\to \R$ such that:
\nn
a) $\phi$ is an exhaustion function (i.e. $\{\phi<c\}\subset\subset X$ for every $c\in\R$),
\nn
b) $\phi$ is strongly plurisubharmonic (its Levi form is positive definite) 
outside a compact subset of $X$.
\nn
It is known (see \cite{Na}) that this is equivalent to the following condition: there exists a proper 
surjective holomorphic map $\pi:X\to Y$ where $Y$ is a normal Stein space with finitely many singular 
points
and there exists a finite subset of $Y$, $B$, such that $\pi$ induces a biholomorphism from 
$X\setminus\pi^{-1}(B)$ to $Y\setminus B$. $\pi^{-1}(B)$ is called the exceptional set of $X$.

A complex space $X$ is called holomorphically convex if for every discrete sequence $\{x_n\}$ in $X$ 
there exists a holomorphic 
function $f:X\to\C$ such that $\lim_{n\to\infty}|f(x)|=\infty$.

We denote by $\Delta$ the unit disk in $\C$,  $\Delta=\{z\in\C: |z|< 1\}$ and for $\epsilon>0$
by $\Delta_{1+\epsilon}$ the disk $\Delta_{1+\epsilon}:=\{z\in\C:|z|<1+\epsilon\}$.

\begin{df}
Suppose that $X$ is a complex space. We say that $X$ satisfies the discrete disk property if
whenever  $f_n:U\to X$ is a sequence of holomorphic functions defined on an open neighborhood 
$U$ of $\overline \Delta$ for which there exists an $\epsilon>0$ and a continuous function 
$\gamma:S^1=\{z\in\C: |z|=1\}\to X$ such that $\Delta_{1+\epsilon}\subset U$, 
$\bigcup_{n\geq 1} f_n(\Delta_{1+\epsilon}\setminus \Delta)$ is 
relatively compact in $X$ and ${f_n}_{|S^1}$ converges uniformly to $\gamma$ we have that 
$\bigcup_{n\geq 1}f_n(\overline \Delta)$ is relatively compact in $X$.
\end{df}

This definition of the disk property using coronae instead of the boundary $\partial \Delta$ of the unit
disk is natural as one sees looking at the example $f_n:\C\to\C$, $f_n(z)=z^n$ which is a sequence that
converges to 0 if $|z|<1$ and diverges otherwise.

For $\epsilon>0$ we define $H_\epsilon\subset \C\times\R$ as $$H_\epsilon=
\Delta_{1+\epsilon}\times[0,1)\bigcup \{z\in\C:1-\epsilon<|z|<1+\epsilon\}\times\{1\}.$$

Having in mind the definition of the continuity principle (see for example \cite{GF}, page 47)
we introduce the following:

\begin{df}
A complex space $X$ is said to satisfy the continuous disk property if whenever $\epsilon$ is a positive number
and $F:H_{\epsilon}\to X$ is a continuous function such that, for every $t\in[0,1)$, 
$F_t:\Delta_{1+\epsilon}\to X$, $F_t(Z)=F(z,t)$, is holomorphic we have that $F(H_{\epsilon_1})$
is relatively compact in $X$ for any $0<\epsilon_1<\epsilon$.
\end{df}

Clearly the discrete  disk property implies the continuous one but as it is shown by the example
of Fornaess (see \cite{Fo}) of an increasing union of Stein open subsets that does not satisfy
the discrete disk property, the discrete disk property is a much stronger condition. 
(It is not difficult to see that an increasing union of open subsets of a complex space $X$, each one 
of them satisfying the continuous disk property, satisfies the continuous disk property. The proof is 
pretty much the same as the proof of Theorem 3.1, page 60 in \cite{GF}.)

 The following theorem was proved in \cite{CD} and \cite{Nap}.
\begin{teo}\label{CDhol}
Let $\pi: X\to  T$  be a proper holomorphic surjective map of complex
spaces, let $t_0 \in  T$ be any point, and denote by $X_{t_0}:=\pi^{-1}(t_0)$ 
the fiber of $\pi$
at $t_0$. Assume that $\dim X_{t_0}=1$. Let $\sigma:\tilde X\to X$ be a covering space and let 
$\tilde X_{t_0}=\sigma^{-1}(X_{t_0})$. If $\tilde X_{t_0}$
 is holomorphically convex, then there exists an open neighborhood $\Omega$ of $t_0$
 such that   $(\pi\circ\sigma)^{-1}(\Omega)$
is holomorphically convex.
\end{teo}

The next result was proved in \cite{Co}.

\begin{prop}\label{CMH}
Let $X$ be an 1-convex manifold with exceptional set set $S$ and $p:\tilde X\to X$ any covering. 
Then there exists
a strongly plurisubharmonic function $\tilde\phi:\tilde X \to [-\infty,\infty)$ such that 
$p^{-1}(S)=\{\tilde \phi=-\infty\}$
and for any open neighborhood $U$ of $S$, the restriction $\tilde\phi_{|\tilde X\setminus p^{-1}(U)}$ 
is an exhaustion function
on $\tilde X\setminus p^{-1}(U)$.
\end{prop}

We recall that a topological space $X$ is called an ENR (Euclidean Neighborhood Retract) if it is 
homeomorphic to a closed subset $X_1$ of $\R^n$ for some $n$ such that there is an open 
neighborhood $V$ of $X_1$ and a continuous retract $\rho:V\to X_1$.  For basic results regarding
ENR's we are referring to \cite{Do}.   It is proved there that a locally compact and locally contractible 
subset of $\R^n$ is and ENR and that if a Hausdorff topological space $X$ is covered by countable family of 
locally compact open
subsets each one of them homeomorphic with a subset of fixed Euclidean space then $X$ is an ENR. In
particular every complex space of bounded dimension is an ENR.

Lemma \ref{CTlema} was proved in \cite{CT}.

\begin{lem}\label{CTlema}
If $X$ is an ENR  and $\{\gamma_n\}_{n\geq 1}$, $\gamma_n:S^1:\to X$, 
is a sequence of null-homotopic loops converging uniformly to $\gamma:S^1\to X$
then $\gamma$ is null-homotopic. Moreover given a covering $p:\tilde X\to X$ there exists liftings 
$\tilde\gamma_n$ for $\gamma_n$ and $\tilde\gamma$ for $\gamma$ such that $\tilde\gamma_n$ converges
uniformly to  $\tilde \gamma$ 
\end{lem}

This lemma can be applied in particular if $X$ is a complex space (by the above discussion).

\begin{df}
Let $L$ be a connected 1-dimensional complex space and $\cup L_i$ be its decomposition into 
irreducible components. $L$ is called an infinite Nori string if all $L_i$ are compact and $L$ is 
not compact
\end{df}

It is clear that a 1-dimensional complex space is holomorphically convex if and only if it does 
not contain an infinite Nori string.

We recall that a compact complex curve is called rational if its normalization is ${\mathbb P}^1$.
\section{The Results}

\begin{lem}\label{holc1}
If $L$ is 1-dimensional complex space that does not contain an infinite Nori string of rational
curves then $L$ has a holomorphically convex covering space.
\end{lem}
\begin{proof} 
\textbf{Step 1.} We assume that $L$ is connected and all its irreducible components are compact not rational.
Let $L=\bigcup_{i\in I}L_i$ be the decomposition of $L$ into irreducible components and
$A=\{x\in L:\exists i_1,i_2\in I,\ i_1\neq i_2, \text{ such that }x\in L_{i_1}\cap L_{i_2}\}$.
Let $p_i:\tilde L_i\to L_i$ be the universal cover of $L_i$. Since $L_i$ is not rational 
for all $i\in I$ we have that $\tilde L_i$ is Stein. 
For each $x\in A\cap L_i$, we choose a bijection $\tau(x,i): \Z\to p_i^{-1}(x)$ and we write
$\tau(x,i)(k)=(k,x,i)$.

On the disjoint union
$S:=\bigsqcup_{i\in I}\tilde L_i$ we define the projection $\tilde p:S\to L$ induced by
$p_i$, $i\in I$. We also define
an equivalence relation on $S$ by defining the equivalence classes as follows: if $s\in S$ is such that
$\tilde p(s)\not\in A$ then $\hat s=\{s\}$. Note that on $\tilde p^{-1}(A)$  we have a map $\zeta:\tilde p^{-1}(A)\to \Z$ 
which is nothing else then the projection on the first component. 
Then if $\tilde p(s)\in A$ and $\zeta(s)=k$ we set
$\hat s=\{(k,x,i):x\in L_i,i\in I,\tilde p(s)=x\}$. We let $\tilde L$ to be the quotient space of $S$ and $p:\tilde L\to L$
the application induced by $\tilde p$. It is not difficult to see that $p$ is a covering and $\tilde L$
is Stein.

\textbf{Step 2.} All irreducible components of $L$ are compact and not rational. In this case we apply the
above procedure to each connected component.

\textbf{Step 3.} The general case. Let $L'$ be the union of all irreducible components of $L$ that 
are compact and not rational and $\{C_j\}_{j\in J}$ the connected components of $L'$. For each $j\in J$
let $p_j:\tilde C_j\to C_j$ the Stein coverings obtained at the first step. Let also $L''$ be the union of the other
irreducible components of $L$. We set $\Gamma_j=L''\cap C_j$. (Note that $\Gamma_j$ might be empty.)
If $\Gamma_j\neq\emptyset$ we consider in $\tilde C_j$ infinitely many disjoint copies $\{\Gamma_j^\alpha:\alpha\in\Z\}$
of $\Gamma_j$. (Each $\Gamma_j^\alpha$ is in bijection via $p_j$ with $\Gamma_j$.) We glue now infinitely many copies $\{L''_\alpha:\alpha\in\Z\}$ of $L''$ to $\tilde C_j$ on $\{\Gamma_j^\alpha:\alpha\in\Z\}$ and we obtain in this way
a holomorphically convex covering of $L$.

\end{proof}

\begin{lem}\label{convlift}
Suppose that $\tilde X$ and $X$ are Hausdorff topological spaces and $p:\tilde X\to X$ is a covering. 
Let 
$\tilde \gamma_n:S^1\to\tilde X$, $\gamma_n:S^1\to X$, $n\geq 1$, 
$\tilde \gamma:S^1\to\tilde X$, $\gamma:S^1\to X$ be 
continuous functions such that $\gamma_n=p\circ\tilde\gamma_n$, $\gamma=p\circ\tilde\gamma$
 and $a\in S^1$ be a fixed point. If $\gamma_n$ converges uniformly to 
$\gamma$ and 
$\tilde \gamma_n(a)$ converges to $\tilde \gamma(a)$ then $\tilde\gamma_n$ converges 
uniformly to $\tilde\gamma$.
\end{lem}
\begin{proof}
Let $\Omega=\{t\in S^1:\lim_{n\to\infty}\tilde \gamma_n(t)=\tilde\gamma(t)\}$. We will show first that 
$\Omega=S^1$. As $a\in\Omega$ we have that $\Omega\neq\emptyset$. We prove that $\Omega$ is open in $S^1$.
Let $t_0\in\Omega$ and let $V_0$ a neighborhood of $\gamma(t_0)$ which is evenly covered by $\bigcup\tilde V_k$.
Let $\Omega_0$ a connected open neighborhood of $t_0$ such that for $n$ large enough $\gamma_n(\Omega_0)\subset V_0$
and $\gamma(\Omega_0)\subset V_0$. It follows that $\tilde\gamma_n(\Omega_0)\subset \bigcup \tilde V_k$.
As  $\tilde\gamma_n(\Omega_0)$  and $\tilde\gamma(\Omega_0)$ are connected we have that for each $n$ 
there exists $k_n$ such that $\tilde\gamma_n(\Omega_0)\subset \tilde V_{k_n}$ and there exists $k_0$ 
such that $\tilde\gamma(\Omega_0)\subset \tilde V_{k_0}$. 
However we assumed that $t_0\in\Omega$ and therefore for $n$ large enough $\tilde\gamma_n(t_0)\in  \tilde V_{k_0}$,
hence $V_{k_n}=V_{k_0}$. Since $p:V_{k_0}\to V_0$ is a homeomorphism we conclude that
$\tilde\gamma_n$ converges uniformly to $\tilde\gamma$ on $\Omega_0$. This shows that $\Omega_0\subset \Omega$ and
hence $\Omega$ is open. 

Next we show in a similar fashion that $\Omega$ is closed. Let $t_0\in S^1\setminus \Omega$. As before 
we choose $V_1$ a neighborhood of $\gamma(t_0)$ which is evenly covered by $\bigcup \tilde V_k$,
$\Omega_0$ a connected open neighborhood of $t_0$ such that for $n$ large enough $\gamma_n(\Omega_0)\subset V_0$,
$k_n$ and $k_0$ such that 
$\tilde\gamma_n(\Omega_0)\subset \tilde V_{k_n}$ and $\tilde\gamma(\Omega_0)\subset \tilde V_{k_0}$.
If $V_{k_n}=V_{k_0}$ for every $n$ we would have that  $\tilde\gamma_n$ converges uniformly to $\tilde\gamma$
on $\Omega_0$ hence at $t_0$ and this would contradict $t_0\not\in\Omega$. There exists then a subsequence 
$\tilde\gamma_{n_p}$ such that $\tilde\gamma_{n_p}(\Omega_0)\cap\tilde V_{k_0}=\emptyset$ and from here we get that
$\Omega_0\subset S^1\setminus\Omega$. Hence $\Omega$ is closed and therefore $\Omega=S^1$. 

Note that when we proved that $\Omega$ is open we actually proved that each
$t_0\in \Omega=S^1$ has an open neighborhood $\Omega_0$ such that on this neighborhood
$\tilde\gamma_n$ converges uniformly to $\tilde\gamma$. The compactness of $S^1$ implies then that
$\tilde\gamma_n$ converges uniformly to $\tilde\gamma$ on $S^1$.
\end{proof}

\begin{lem}\label{dischol}
Suppose that $S$ is a complex space that has a holomorphically convex covering 
$p:\tilde S\to S$. If $\Omega\subset \C$ is an open neighborhood of  $\overline\Delta$ and
$f_n:\Omega\to S$ is a sequence of holomorphic functions which converges uniformly on 
$\{z\in\C: |z|=1\}$ then  $\bigcup_n f_n(\overline\Delta)$ is relatively compact in $S$.
(In particular $S$ satisfies the discrete disk property.)
\end{lem}

\begin{proof}
Let $\gamma:\{z\in\C: |z|=1\}\to S$ be the limit of $\{f_n\}$ on $\{z\in\C: |z|=1\}$. 
It follows from Lemma \ref{CTlema} that $\gamma$ is null-homotopic. 
We choose a point $a\in \{z\in\C: |z|=1\}$.
Let $\epsilon$ be a positive number 
such that $ \Delta_{1+\epsilon}\subset \Omega$, $\tilde f_n: \Delta_{1+\epsilon}\to  \tilde S$
and $\tilde\gamma:\{z\in\C: |z|=1\}\to \tilde S$ be liftings of $f_n$ and $\gamma$ respectively.
We choose $\tilde f_n$ and $\tilde\gamma$ such that $\tilde f_n$ converges uniformly to $\tilde \gamma$
on $\{z\in\C: |z|=1\}$. This is possible by 
Lemma \ref{CTlema}. As $\tilde S$ is holomorphically convex this implies that 
$\bigcup_n \tilde f_n(\overline\Delta)$ is relatively compact in $\tilde S$  and therefore 
$\bigcup_n f_n(\overline\Delta)$  is relatively compact in $S$.
\end{proof}

\begin{teo}\label{main}
Let $X$ be a 1-convex surface and $p:\tilde X\to X$ be a covering map. If $\tilde X$ does not contain
an infinite Nori string of rational curves then $\tilde X$ satisfies the discrete disk property.
\end{teo}

\begin{proof}
Let $L$ be the exceptional set of $X$. Without loss of generality we may assume that $L$ is connected.
Let $Y$ be a normal Stein complex space of dimension 2 and 
$\pi:X\to Y$ a proper surjective holomorphic map such  that $\pi(L)=\{y_0\}$ and 
$\pi:X\setminus L\to Y\setminus\{y_0\}$ is a biholomorphism. 
We may assume that $Y$ is an analytic closed subset of $\C^N$.

We put $\tilde L:=p^{-1}(L)$.
Since we assumed that $\tilde L$ does not contain an infinite Nori string of rational curves
it follows from Lemma \ref{holc1} that $\tilde L$ has a holomorphically convex covering $\hat p:\hat L\to \tilde L$.
We choose $U_1$ an open neighborhood of $L$ in $X$ that has a continuous deformation retract on $L$.
We let $\tilde U_1=p^{-1}(U_1)$ and then $\tilde U_1$ has a continuous deformation retract on $\tilde L$,
$\rho:\tilde U_1\to\tilde L$. By considering the fiber product of $\rho$ and $\hat p$ we obtain
a covering $\hat U_1$ of $\tilde U$ such that the pull-back of $\tilde L$ is $\hat L$.
We apply then Theorem \ref{CDhol} and we deduce that there exists $U\subset Y$ an open neighborhood of 
$y_0$ such that $(\pi\circ p)^{-1}(U)$ has a 
holomorphically convex covering.

Let $f_n:\Delta_{1+\epsilon}\to \tilde X$, $n\geq 1$ be a sequence of holomorphic functions such that 
$\bigcup_{n\geq 1} f_n(\Delta_{1+\epsilon}\setminus \Delta)$ is 
relatively compact in $\tilde X$ and ${f_n}_{|S^1}$ is uniformly convergent. 
We argue by contradiction and we assume that $\cup_{n\geq 1} f_n(\overline\Delta)$ is not relatively 
compact in $\tilde X$  and hence $\cup_{n\geq 1} f_n(\Delta_{1+\epsilon})$ is 
not relatively compact. 
By passing to a subsequence we may assume that $\cup_{k\geq 1} f_{n_k}(\Delta_{1+\epsilon})$  
is not relatively compact 
in $\tilde X$ for every subsequence $\{f_{n_k}\}_k$ of $\{f_n\}_n$.

Since  $\bigcup(\pi\circ p\circ f_n)(\Delta_{1+\epsilon}\setminus\Delta)$ is relatively compact in $Y$, hence in $\C^N$,
by the maximum modulus principle we have that $\bigcup(\pi\circ p\circ f_n)(\Delta_{1+\epsilon})$ 
 is relatively compact in $Y$ and then it follows that
there exists a subsequence of $\{f_n\}_n$, $\{f_{n_k}\}_k$ such   that $\pi\circ p\circ f_{n_k}$ 
converges uniformly on compacts to a 
holomorphic function $g:D_\epsilon\to Y$. Without loss of generality we can assume that 
$\pi\circ p\circ f_n$ converges uniformly on compacts to $g$. We distinguish three cases.

\textbf{Case 1.} $g\equiv y_0$. Then 
$(\pi\circ p\circ f_n)(\overline\Delta_{1+\epsilon_1})\subset U$
for $\epsilon_1\in(0,\epsilon)$ and $n$ large enough.  Therefore we can apply Lemma \ref{dischol}.

\textbf{Case 2.} $g(z)\neq y_0$ for every $z\in\Delta_{1+\epsilon}$. Then 
there exists a 
neighborhood $V$ of $y_0$ and $\epsilon_1\in(0,\epsilon)$ such that 
$g(\overline \Delta_{1+\epsilon_1})\cap \overline V=\emptyset$. 
Then for $n$ large enough we get that  
$\pi\circ p\circ f_n(\overline \Delta_{1+\epsilon_1})\cap \overline V=\emptyset$
and hence $f_n(\overline \Delta_{1+\epsilon_1})\cap (\pi\circ p)^{-1} (V)=\emptyset$.
We consider now a plurisubharmonic function $\tilde \phi$ on $\tilde X$ with the properties  
given in Proposition \ref{CMH}. In particular its restriction to 
$\tilde X\setminus  (\pi\circ p)^{-1} (V)$ is an exhaustion. 
Applying the maximum principle to $\tilde \phi\circ f_n$
we obtain immediately that $\bigcup f_n(\overline \Delta)$ is relatively compact.

\textbf{Case 3.} $g\not\equiv y_0$ and $y_0\in g( \Delta_{1+\epsilon})$. Then $g^{-1}(y_0)$ is a 
(non-empty) discrete subset of 
$\Delta_{1+\epsilon}$ and therefore there exists
$\epsilon_1\in(0,\epsilon)$ such that  $g^{-1}(y_0)\cap \{z\in\C:|z|=1+\epsilon_1\}=\emptyset$. 
We set
 $g^{-1}(y_0)\cap \{z\in\C:|z|<1+\epsilon_1\}=\{\lambda_1,\lambda_2,\dots,\lambda_s\}$. 
Let $r_1,\dots,r_s$ be positive numbers such that the discs 
$\overline \Delta_j=\{z\in\C:|z-\lambda_j|\leq r_j\}$, $j=1,2,\dots, s$ 
are pairwise disjoint, $\overline \Delta_j\subset \Delta_{1+\epsilon_1}$ and 
$g(\overline \Delta_j)\subset U$.
Let $V\subset Y$ be an open neighborhood of $y_0$ such that $\overline V\subset U$ and 
$g(\{z\in  \overline\Delta_{1+\epsilon_1}:|z-\lambda_j|\geq r_j\ \forall j=1,..,s\})\cap\overline V
=\emptyset$. Then for $n$ large enough we have that 
$(\pi\circ p\circ f_n)(\{z\in \overline\Delta_{1+\epsilon_1}:|z-\lambda_j|\geq r_j\ \forall j=1,..,s\})
\cap\overline V=\emptyset$.
We apply again Proposition \ref{CMH} and we obtain a plurisubharmonic function
$\tilde\phi:\tilde X \to [-\infty,\infty)$ such that $\tilde L=\{\tilde \phi=-\infty\}$
$\tilde\phi_{|\tilde X\setminus (\pi\circ p)^{-1}(V)}$ is an exhaustion function
on $\tilde X\setminus  (\pi\circ p)^{-1}(V)$. Since $\bigcup f_n(\Delta_{1+\epsilon}\setminus\Delta)$
is relatively compact there exists a positive constant $M$ such that $\tilde \phi\circ f_n(z)\leq M$
for every $n$ and every $z\in \Delta_{1+\epsilon}\setminus\Delta$. From the plurisubharmonicity of 
$\tilde\phi$ we have that  $\tilde \phi\circ f_n(z)\leq M$ for every  $z\in \Delta_{1+\epsilon}$, 
hence in particular for $z\in\overline\Delta_{1+\epsilon_1}\setminus\bigcup_{j=1}^s\Delta_j$.
As $f_n(\overline\Delta_{1+\epsilon_1}\setminus\bigcup_{j=1}^s\Delta_j)\subset \tilde X\setminus  
(\pi\circ p)^{-1}(V)$ and $\tilde\phi_{|\tilde X\setminus (\pi\circ p)^{-1}(V)}$ is an exhaustion 
we deduce that $\bigcup f_n(\overline\Delta_{1+\epsilon_1}\setminus\bigcup_{j=1}^s\Delta_j)$
is relatively compact in $\tilde X$. 

For $j=1,\dots,s$ we set $S_j=\{z\in\Delta_{1+\epsilon_1}:|z-\lambda_j|=r_j\}$ and we pick a point
$a_j\in S_j$. As $\bigcup _{n\geq 1}f_n(S_j)$ is relatively compact, by passing to a subsequence
we can assume that each sequence $\{f_n(a_j)\}_n$ is convergent and we denote by $x_j$ its limit. 
We have that $(\pi\circ p\circ f_n)_{|S_j}$ converges uniformly to $g_{|S_j}$ and 
$(\pi\circ p\circ f_n)(S_j)\subset Y\setminus\{y_0\}$, $g(S_j)\subset Y\setminus\{y_0\}$. 
Because $\pi:X\setminus L\to  Y\setminus\{y_0\}$ is a biholomorphism we  deduce that
$(p\circ f_n)_{|S_j}$ converges uniformly to $(\pi^{-1}\circ g)_{|S_j}$. Note now that 
$(p\circ f_n)_{|S_j}$ is in fact a null-homotopic loop. It follows from Lemma \ref{CTlema}
that  $(\pi^{-1}\circ g)_{|S_j}$ is null-homotopic as well and therefore there exists a loop
$\gamma_j:S_j\to \tilde X$ such that $p\circ\gamma_j=(\pi^{-1}\circ g)$ on $S_j$.
From Lemma \ref{convlift} we conclude that $f_n$ converges uniformly to $\gamma_j$ on $S_j$.
Finally Lemma \ref{dischol} implies that $\bigcup f_n(\overline\Delta_j)$ is relatively compact for every
$j=1,\dots,s$.
\end{proof}

\ \nn
$\mathbf{Acknowledgments}:$ \* {\it Both authors were partial supported by CNCSIS Grant PN-II ID\_1185, contract 472/2009.
}

\vspace{1.0cm}
\begin{flushleft}
Mihnea Col\c toiu \newline
Institute of Mathematics of the Romanian Academy\newline
P.O. Box 1-764, Bucharest 014700\newline 
ROMANIA\newline
\emph{E-mail address}: Mihnea.Coltoiu@imar.ro

\

\

Cezar Joi\c ta \newline
Institute of Mathematics of the Romanian Academy\newline
P.O. Box 1-764, Bucharest 014700\newline 
ROMANIA\newline
\emph{E-mail address}: Cezar.Joita@imar.ro
\end{flushleft}

\end{document}